\documentclass[oneside,12pt]{amsart} 
\usepackage{eurosym}
\usepackage{amsfonts}
\usepackage{amsmath,amssymb,amsthm,color}
\usepackage{amsopn}
\usepackage{latexsym}
\usepackage{float}
\usepackage{bm}
\usepackage[utf8]{inputenc}
\usepackage{hhline}
\usepackage[hidelinks]{hyperref}
\usepackage{calc}
\newsavebox\CBox
\newcommand\hcancel[2][0.5pt]{%
  \ifmmode\sbox\CBox{$#2$}\else\sbox\CBox{#2}\fi%
  \makebox[0pt][l]{\usebox\CBox}%
  \rule[0.5\ht\CBox-#1/2]{\wd\CBox}{#1}}
\usepackage{blindtext}
\usepackage{color}
\usepackage{hyperref,url}
\usepackage{float}
\usepackage{hyperref}
\usepackage{enumerate}
\usepackage{enumitem}   
\usepackage{cancel} 
\usepackage{mathrsfs}
\usepackage{colonequals}
\usepackage{pdfpages}

\usepackage{amsmath,amsfonts,amsthm,bm}
\usepackage{mathrsfs}  
\usepackage{xcolor}
\usepackage{amsfonts}
\usepackage{amssymb}
\usepackage{amsmath}
\usepackage{mathtools}
\usepackage{mathrsfs}
\usepackage[normalem]{ulem}  
\numberwithin{equation}{section}
\usepackage[toc,page]{appendix} 
\theoremstyle{definition}  
\usepackage{stackengine}
\usepackage[english]{babel}
\usepackage{amsmath}
\usepackage{graphicx}
\usepackage{epstopdf}
\usepackage{array}
\usepackage{latexsym}
\usepackage[font=small]{caption}
\usepackage[all]{hypcap}
\usepackage{enumitem}
\usepackage[utf8]{inputenc}
\usepackage{physics}
\usepackage{cancel}
\usepackage{amssymb}
\usepackage{mathtools}
\usepackage{xcolor}
\usepackage{framed}
\usepackage{amsthm}
\usepackage{graphicx}
\usepackage{amsmath}
\usepackage{graphicx}
\usepackage{tikz-cd}
\usepackage{blindtext}
\usepackage[toc,page]{appendix}
\usepackage{mathrsfs}
\usepackage{verbatim}
\usepackage{float}
\usepackage{placeins}
\usepackage{mathtools}
\usepackage{multirow}
\usepackage{blkarray}
\usepackage{tabularx,longtable,multirow,subfigure,caption}
\usepackage{soul}
\usepackage{lscape}
\usepackage{todo}
\usepackage{adjustbox}
\usepackage{faktor}
\usepackage{upgreek}
\usepackage{booktabs}
\usepackage{pgf,tikz,pgfplots}
\usepackage{parskip} 
\usepackage{amsaddr}

\title[Generalised Killing Spinors on Three-Dimensional Lie Groups]{Generalised Killing Spinors on \\ Three-Dimensional Lie Groups}
\author[Diego Artacho]
{Diego Artacho} 

\address{Department of Mathematics, Imperial College London \\  Email: d.artacho21@imperial.ac.uk \\ ORCID: 0000-0002-1390-9980 \\ MSC: 	53C27 (Primary), 22E60, 15A66 (Secondary)}

\allowdisplaybreaks

\def\@makechapterhead#1{%
  \vspace*{10\p@}%
  {\parindent \z@ \raggedright \sffamily
    \interlinepenalty\@M
    \Huge\bfseries \thechapter \space\space #1\par\nobreak
    \vskip 30\p@
  }}

\def\@makeschapterhead#1{%
  \vspace*{10\p@}%
  {\parindent \z@ \raggedright
    \sffamily
    \interlinepenalty\@M
    \Huge \bfseries  #1\par\nobreak
    \vskip 30\p@
  }}

\allowdisplaybreaks

\makeatletter
\DeclareRobustCommand\mapstofill{%
  $\m@th
  {\mapstochar}%
  \smash-\mkern-7mu
  \cleaders\hbox{$\mkern-2mu\smash-\mkern-2mu$}\hfill
  \mkern-7mu
  \mathord\rightarrow
  $%
}
\makeatother

\makeatletter
  \def\subsection{\@startsection{subsection}{2}%
    \z@{.5\linespacing\@plus.7\linespacing}
  {.5\baselineskip}%
    {\normalfont\flushleft\scshape}%
  }
  \makeatother

  \makeatletter
  \def\subsubsection{\@startsection{subsubsection}{2}%
    \z@{.5\linespacing\@plus.7\linespacing}
  {.5\baselineskip}%
    {\normalfont\flushleft\scshape}%
  }
  \makeatother
  
  \usepackage{etoolbox}
  \patchcmd{\section}{\scshape}{\bfseries}{}{}
  \patchcmd{\subsection}{\scshape}{\bfseries}{}{}
  \patchcmd{\subsubsection}{\scshape}{\bfseries}{}{}
  \makeatletter
  \renewcommand{\@secnumfont}{\bfseries}
  \makeatother

\makeatletter
\def\@tocline#1#2#3#4#5#6#7{\relax
  \ifnum #1>\c@tocdepth 
  \else
    \par \addpenalty\@secpenalty\addvspace{#2}%
    \begingroup \hyphenpenalty\@M
    \@ifempty{#4}{%
      \@tempdima\csname r@tocindent\number#1\endcsname\relax
    }{%
      \@tempdima#4\relax
    }%
    \parindent\z@ \leftskip#3\relax \advance\leftskip\@tempdima\relax
    \rightskip\@pnumwidth plus4em \parfillskip-\@pnumwidth
    #5\leavevmode\hskip-\@tempdima
      \ifcase #1
       \or\or \hskip 1em \or \hskip 2em \else \hskip 3em \fi%
      #6\nobreak\relax
    \hfill\hbox to\@pnumwidth{\@tocpagenum{#7}}\par
    \nobreak
    \endgroup
  \fi}
\makeatother

\usepackage{etoolbox}

\makeatletter
\patchcmd{\@abstract}{\large}{\large}{}{}
\makeatother

\newcommand{\R}[0]{\mathbb{R}} 
 
\newcommand{\N}[0]{\mathbb{N}} 
\newcommand{\C}[0]{\mathbb{C}}

\DeclareMathOperator{\vecspan}{span}

\DeclareMathOperator{\Id}{Id}

\DeclareMathOperator{\GL}{GL}

\DeclareMathOperator{\Spin}{Spin}

\DeclareMathOperator{\Ric}{Ric}

\DeclareMathOperator{\Mat}{M}

\DeclareMathOperator{\End}{End}

\theoremstyle{definition}
\newtheorem{deff}{Definition}[section]

\newtheorem{rem}[deff]{Remark}

\theoremstyle{definition}

\newtheorem{thmm}[deff]{Theorem}
\newtheorem*{thmm*}{Theorem}
\newtheorem{lemma}[deff]{Lemma}

\newtheoremstyle{named}{}{}{\itshape}{}{\bfseries}{.}{.5em}{\thmnote{#3's }#1}
\theoremstyle{named}

\begin{document}

\keywords{Generalised Killing spinors, Heisenberg groups, Lie groups, spin geometry}

\maketitle  
\begin{abstract} 
    \normalsize 
    We present a complete classification of invariant generalised Killing spinors on three-dimensional Lie groups. We show that, in this context, the existence of a non-trivial invariant generalised Killing spinor implies that all invariant spinors are generalised Killing with the same endomorphism. Notably, this classification is independent of the choice of left-invariant metric. To illustrate the computational methods underlying this classification, we also provide the first known examples of homogeneous manifolds admitting invariant generalised Killing spinors with $n$ distinct eigenvalues for each $n > 4$. 
\end{abstract}
 
\tableofcontents

\clearpage

\section{Introduction} \label{sec:intro}

The study of special spinors on Riemannian manifolds is of great geometric interest. For instance, the existence of \textit{parallel} spinors characterises manifolds with special holonomy \cite{W89}. In this paper, we focus on \textit{generalised Killing spinors}, which are defined as follows. Let $(M,g)$ be a Riemannian manifold equipped with a spin structure, and let $\Sigma M$ be the associated spinor bundle. A spinor $\psi \in \Gamma \left( \Sigma M \right)$ is said to be generalised Killing if there exists a $g$-symmetric endomorphism field $A \in \Gamma \left( \End(TM) \right)$ such that, for every vector field $X \in \Gamma\left(TM\right)$, 
\begin{equation} \label{eq:genkil}
\widetilde{\nabla}^{g}_X \psi = A(X) \cdot \psi \, , 
\end{equation}
where $\widetilde{\nabla}^{g}$ is the spin connection and $\cdot$ denotes Clifford multiplication \cite{KF00}. Note that this generalises the notion of parallel ($A=0$) and Killing ($A = \lambda \Id$, $\lambda \in \mathbb{C}$) spinors. The study of generalised Killing spinors is motivated by the fact that a parallel spinor on an $(n+1)$-dimensional Riemannian manifold induces a generalised Killing spinor on any hypersurface, with the endomorphism field $A$ corresponding to twice the second fundamental form. Conversely, in the real analytic setting, an $n$-dimensional Riemannian manifold with a generalised Killing spinor can be embedded into an $(n+1)$-dimensional Riemannian manifold with a parallel spinor \cite{AMM}. 

These spinors have been extensively studied in the context of submanifold theory \cite{F98,BGM,LR}, special geometries and $G$-structures \cite{ContiSalamon}. However, their study proves more difficult than that of parallel or Killing spinors. There are still many open problems concerning them; for example, a classification is out of reach, even for the case of simple spaces like the round sphere $S^3$ -- see \cite{MS14}. 

In this paper, we consider left-invariant metrics on three-dimensional Lie groups, and restrict our attention to those spinors which are left-invariant. The problem can then be treated at the level of Lie algebras. Bianchi \cite{bianchi} classified real three-dimensional Lie algebras into nine types: seven individual ones ($\mathrm{Bian}(\text{I})$ -- $\mathrm{Bian}(\text{V})$, $\mathrm{Bian}(\text{VIII})$ and $\mathrm{Bian}(\text{IX})$) and two infinite one-parameter families ($\mathrm{Bian}(\text{VI}_{x})$ and $\mathrm{Bian}(\text{VII}_{y})$, with $-1 \leq x < 1$, $x \neq 0$ and $y \geq 0$). We use the more modern treatment \cite{M63} as a reference for explicit realisations of these Lie algebras -- see Appendix \ref{appendix}. 

The main result is given in Theorem \ref{thm:3dim}, and is summarised in Table \ref{table:3-dim}. Remarkably, the existence of a non-trivial invariant generalised Killing spinor implies that all invariant spinors are generalised Killing with the same endomorphism. This makes it possible to carry out computations that depend only on the metric, and not on the choice of a spinor. Moreover, the question whether all invariant spinors are generalised Killing is independent of the choice of left-invariant metric. 

We say that a generalised Killing spinor $\psi$ \textit{has $r$ distinct eigenvalues} if the endomorphism field $A$ does. If $r=1$, $\psi$ is a Killing spinor. On homogeneous manifolds, invariant examples with two \cite{BFGK,AHL}, three \cite{AFS,3adsasaki} and four \cite{AHL} distinct eigenvalues are known. In Section \ref{sec:heisenberg}, to illustrate the methods required to prove Theorem \ref{thm:3dim}, we prove the following -- see Theorem \ref{thmm:main}: 

\begin{thmm*}
    For each $n \in \N$, the $n$-th real Heisenberg group $H_{2n+1}$ has a left-invariant metric which carries an invariant generalised Killing spinor with $(n+1)$ distinct eigenvalues.
\end{thmm*} 

These are the first examples of invariant generalised Killing spinors on a homogeneous manifold with more than four distinct eigenvalues. 

\begin{table}[h!]
    \begin{center}
\begin{tabular}{@{\hspace{10mm}}ccc@{\hspace{10mm}}}
    \toprule
     $\mathfrak{g}$ & $\dim\left(\Sigma_{\text{inv}}^{GK}\right)$ & $r$ \\
 \toprule
 $\mathrm{Bian}(\text{I})$ & $2$ & $1$ \\
 \midrule
 $\mathrm{Bian}(\text{II})$ & $2$ & $2$ \\
 \midrule
 $\mathrm{Bian}(\text{III})$, $\mathrm{Bian}(\text{IV})$, $\mathrm{Bian}(\text{V})$ & $0$ & $-$   \\   
 \midrule
 \multirow{2}{*}{$\mathrm{Bian}(\text{VI}_{x})$} & $2$, if $x = -1$ & $3$ \\
 & $0$, if $x \neq -1$ & $-$ \\
 \midrule
 \multirow{2}{*}{$\mathrm{Bian}(\text{VII}_{y})$} & $2$, if $y = 0$ & $3$ \\
 & $0$, if $y \neq 0$ & $-$ \\
 \midrule
 $\mathrm{Bian}(\text{VIII})$, $\mathrm{Bian}(\text{IX})$ & $2$ & $3$ \\
 \bottomrule
 \end{tabular}
 \caption{Maximum (and generic) number of distinct eigenvalues, denoted by $r$, of invariant generalised Killing spinors on a connected three-dimensional Lie group with Lie algebra $\mathfrak{g}$ equipped with \textit{any} orientation and  \textit{any} left-invariant metric. $\Sigma_{\text{inv}}^{GK}$ denotes the complex dimension of the vector space of invariant generalised Killing spinors. }
 \label{table:3-dim}
    \end{center}
\end{table}

\FloatBarrier

\section{Preliminaries}\label{section:preliminaries} 

\subsection{Differential forms approach to the spin representation} \label{sec:diff_forms_approach}

For a more general introduction to this approach to the spin representation, see \cite{AHL}. Let $n\in\mathbb{N}$, and let $G$ be a connected $(2n+1)$-dimensional Lie group. Left-invariant Riemannian metrics on $G$ are in bijective correspondence with inner products on its Lie algebra $\mathfrak{g}$. Let $B$ be such an inner product, and let $g$ be the corresponding metric on $G$. Let $(e_1,\dots,e_{2n+1})$ be a $B$-orthonormal basis of $\mathfrak{g}$. 

First, consider the complexification $\mathfrak{g}^{\C}$ of the real vector space $\mathfrak{g}$. It is clear that 
\[ 
    \mathfrak{g}^{\C} = \vecspan_{\C} \left\{ e_1 \right\} \oplus L \oplus L'  \, , 
\]
where 
\[  
    L = \vecspan_{\C} \left\{ x_p = \frac{1}{\sqrt{2}} \left( e_{2p} - i e_{2p+1} \right) \right\}_{p=1 , \dots , n}\,  , \quad L' = \vecspan_{\C} \left\{ y_p = \frac{1}{\sqrt{2}} \left( e_{2p} + i e_{2p+1} \right) \right\}_{p=1 , \dots , n} \, . 
\]
Let 
\[ 
    \Sigma = \Lambda^{\bullet} L' \, . 
\]
Note that we can let $\mathfrak{g}^{\C}$ act on $\Sigma$ by extending the following: for each $\eta \in \Sigma$ and $1 \leq p \leq n$, 
\[ 
    e_1 \cdot \eta = i\left( \eta_{\text{even}} - \eta_{\text{odd}} \right) \, , \qquad x_p \cdot \eta = i \sqrt{2} x_p \lrcorner \eta \, , \qquad y_p \cdot \eta = i \sqrt{2} y_p \wedge \eta \, . 
\] 
This action of $\C^n \cong \mathfrak{g}^{\C}$ on $\Sigma$ extends to a non-trivial action of the complex Clifford algebra $\C\rm{l}(2n+1)$ on $\Sigma$. And the dimension of $\Sigma$ is $2^n$. Hence, the restriction of this action to $\Spin(2n+1)$ is isomorphic to the spin representation. Note that, for each $1 \leq p \leq n$ and $\eta \in \Sigma$, 
\[ 
    \qquad e_1 \cdot \eta = i\left( \eta_{\text{even}} - \eta_{\text{odd}} \right) \, , \qquad e_{2p} \cdot \eta = i \left( x_p \lrcorner \eta + y_p \wedge \eta \right) \, , \qquad e_{2p+1} \cdot \eta = \left( y_p \wedge \eta - x_p \lrcorner \eta \right) \, . 
\] 

$G$ has a unique $G$-invariant spin structure in the sense of \cite{DKL}. Its associated spinor bundle is given by the trivial bundle 
\[ 
    G \times \Sigma \, ,  
\] 
and sections of this bundle correspond to smooth maps $G \to \Sigma$. $G$ acts on the space of spinors as follows: for $x \in G$ and a spinor $\psi$, 
\[ 
    (x \cdot \psi)(y) = \psi(x^{-1} \cdot y) . 
\] 
A spinor $\psi$ is said to be invariant if it is stabilised by all elements of $G$, \textit{i.e.}, if $\psi$ is a constant map. Hence, elements of $\Sigma$ correspond to invariant spinors.

\subsection{Invariant connections and Nomizu maps}

For an extensive treatment of this topic, see \cite{ANT}. We will only need the following result, which lets us translate the generalised Killing equation \ref{eq:genkil} into an algebraic equation.   

\begin{lemma} \label{lemma:nomizu}
Let $m \in \mathbb{N}$ and let $G$ be a connected $m$-dimensional Lie group. Let $B$ be an inner product on its Lie algebra $\mathfrak{g}$, let $g$ be the corresponding left-invariant metric on $G$, and let $\nabla^g$ be the Levi-Civita connection of $g$. Define the Nomizu map of $\nabla^g$ to be the unique linear map $\Uplambda \colon \mathfrak{g} \to \mathfrak{so}(m)$ satisfying 
\[ 
\Uplambda(X)(Y) = \frac{1}{2} \left[ X,Y \right] + U(X,Y) \, , 
\] 
where 
\[ 
B \left( U(X,Y) , W \right) = \frac{1}{2} \left( B \left( \left[W,X\right] , Y \right) + B \left( X , \left[W,Y\right] \right)  \right) \, . 
\]
Let $\widetilde{\Uplambda} = f \circ \Uplambda$, where 
\begin{align*} 
f \colon &\mathfrak{so}(m) \overset{\cong}{\longrightarrow} \mathfrak{spin}(m) \\
&e_i \wedge e_j \mapsto \frac{1}{2} e_i \cdot e_j \, , 
\end{align*}
and $e_i \wedge e_j$ represents the element of $\mathfrak{so}(m)$ that maps $e_i$ to $e_j$, $e_j$ to $-e_i$ and $e_k$ to 0 for all $k \neq i,j$. Let $\psi \in \Sigma$ be an invariant spinor. Then, for each $X \in \mathfrak{g}$, 
\[ 
    \widetilde{\nabla}^{g}_{\hat{X}} \psi = \widetilde{\Uplambda}(X) \cdot \psi \, , 
\]
where $\hat{X}$ is the left-invariant vector field determined by $X$, $\cdot$ denotes the differential of the spin representation introduced in \ref{sec:diff_forms_approach}. \qed 
\end{lemma}

\section{The Heisenberg groups} \label{sec:heisenberg}

In this section we exhibit the first known examples of invariant generalised Killing spinors with an arbitrary number of distinct eigenvalues on a homogeneous manifold. For each $n \in \mathbb{N}$, we will find a left-invariant metric on the $n$-th Heisenberg group $H_{2n+1}$ which carries an invariant generalised Killing spinor with $(n+1)$ distinct eigenvalues. The calculations carried out in this section are crucial for Section \ref{sec:3dim}, as they exemplify the methods required there. 

We start by recalling the definition of the Heisenberg groups. 

\begin{deff}
    Let $n \in \N$. The $n$-th (real) Heisenberg group $H_{2n+1}$ is defined as 
    \[ 
        H_{2n+1} = \left\{ \begin{pmatrix} 1 & v^{T} & x \\ 0 & \Id_n & w \\ 0 & 0 & 1 \end{pmatrix} \in \Mat_{n+2} \left( \R \right) \colon v, w \in \R^n , \, x \in \R \right\} \, ,  
    \]
    where we think of elements $v \in \R^n$ as column vectors and $v^{T}$ denotes the transpose of $v$. 
\end{deff}
With the obvious topology and smooth structure, $H_{2n+1}$ is a Lie group which is diffeomorphic to $\R^{2n+1}$. The Lie algebra of $H_{2n+1}$ is 
\[ 
\mathfrak{h}_{2n+1} = \vecspan_{\R} \left\{ \hat{Z}, \, \hat{E}_p , \, \hat{F}_p \right\}_{p = 1 , \dots , n} \, , 
\]
where 
\[ 
    \hat{Z} = \begin{pmatrix} 1 & 0^{T} & 1 \\ 0 & \Id_n & 0 \\ 0 & 0 & 1 \end{pmatrix} \, , \qquad \hat{E}_p = \begin{pmatrix} 1 & e_p^{T} & 0 \\ 0 & \Id_n & 0 \\ 0 & 0 & 1 \end{pmatrix} \, , \qquad \hat{F}_p = \begin{pmatrix} 1 & 0^{T} & 0 \\ 0 & \Id_n & e_p \\ 0 & 0 & 1 \end{pmatrix} \, 
\] 
and $e_p$ denotes the $p$-th vector of the standard basis of $\R^n$. These generators of $\mathfrak{h}_{2n+1}$ satisfy, for $1 \leq p,q \leq n$, 
\[ 
 \left[\hat{E}_p , \hat{F}_q \right] = \delta_{p,q} \hat{Z}  \, , \qquad \left[\hat{E}_p , \hat{Z} \right] = \left[\hat{F}_p , \hat{Z} \right] = \left[\hat{E}_p , \hat{E}_q \right] = \left[\hat{F}_p , \hat{F}_q \right] = 0 \, ,  
\]
where $\delta$ is the Kronecker delta. Observe that this Lie algebra has one-dimensional centre. We will fix the orientation determined by the ordered basis 
\begin{equation} \label{eq:basis_h} 
    \left( \hat{Z} , \hat{E}_1 , \hat{F}_1 , \dots , \hat{E}_n , \hat{F}_n \right) \, . 
\end{equation} 
Now, suppose that $g$ is a left-invariant metric on $H_{2n+1}$ determined by an inner product $B$ on $\mathfrak{h}_{2n+1}$. Then, by applying the Gram-Schmidt process to \eqref{eq:basis_h}, we obtain a positively oriented $B$-orthonormal basis of $\mathfrak{h}_{2n+1}$ which satisfies the same Lie bracket relations as 
\begin{equation} \label{eq:basis_h_2}
    \left(  Z \coloneqq \hat{Z} \, , E_1 \coloneqq a_1 \hat{E}_1 , \, F_1 \coloneqq  \hat{F}_1 \, , \dots \, , E_n \coloneqq a_n \hat{E}_n , \, F_n \coloneqq \hat{F}_n  \right) \, , 
\end{equation}
for some $(a_1 , \dots , a_n) \in \R_{>0}^n$. This allows us to obtain an expression for the Nomizu map of the Levi-Civita connection of $g$. 

\begin{lemma} \label{lemma:h_nomizu}
The Nomizu map $\Uplambda \colon \mathfrak{h}_{2n+1} \to \mathfrak{so}(\mathfrak{h}_{2n+1})$ of the Levi-Civita connection of $g$ is given by 
\[ \Uplambda \left( Z \right) = -\frac{1}{2} \sum_{p=1}^{n} a_p E_p \wedge F_p\, , \quad \Uplambda \left( F_{p} \right) = -\frac{1}{2} a_p E_p \wedge Z \, , \quad \Uplambda \left( E_{p} \right) = \frac{1}{2} a_p F_p \wedge Z\, . \] 
\end{lemma}
\begin{proof} 
Recall from \cite{nomizu,ANT}\footnote{There is a sign mistake in \cite{nomizu}, which is fixed in \cite{ANT}. } that the Nomizu map $\Uplambda$ of the Levi-Civita connection of the invariant metric $g$ on $H_{2n+1}$ induced by the inner product $B$ on $\mathfrak{h}_{2n+1}$ satisfies, for every $X,Y \in \mathfrak{h}_{2n+1}$, 
\[ 
\Uplambda(X)(Y) = \frac{1}{2} \left[ X,Y \right] + U(X,Y) \, , 
\] 
where the bilinear map $U$ is determined by imposing, for all $W \in \mathfrak{h}_{2n+1}$, 
\[ 
B \left( U(X,Y) , W \right) = \frac{1}{2} \left( B \left( \left[W,X\right] , Y \right) + B \left( X , \left[W,Y\right] \right)  \right) \, . 
\]
Fix $1 \leq p \leq n$. Let us first compute $\Uplambda \left( E_{p} \right)$. Note that, for $1 \leq q \leq n$,  
\begin{align*} 
    \Uplambda \left( E_{p} \right) (E_q) &= \frac{1}{2} \left[ E_p,E_q \right] + U(E_p,E_q) = 0 \, , \\
    \Uplambda \left( E_{p} \right) (F_q) &= \frac{1}{2} \left[ E_p,F_q \right] + U(E_p,F_q) = \frac{1}{2} \left[ E_p,F_q \right] = \\ 
     &=\frac{1}{2} a_p \left[ \hat{E}_p , \hat{F}_q \right] = \delta_{p,q}\frac{1}{2} a_p Z \, , \\
     \Uplambda \left( E_{p} \right) (Z) &= \frac{1}{2} \left[ E_p,Z \right] + U(E_p,Z) = U(E_p,Z) = \\
     &= \frac{1}{2} B \left(  \left[F_p,E_p\right] , Z \right) F_p = -\frac{1}{2} a_p F_p \, . 
\end{align*}
This proves that $\Uplambda \left( E_{p} \right)$ is as claimed. A similar computation yields that $\Uplambda \left( F_{p} \right)$ is as claimed. Finally, let us compute $\Uplambda \left( Z \right)$. For each $1 \leq p \leq n$, we obtain  
\begin{align*} 
    \Uplambda \left( Z \right) (E_p) &= \frac{1}{2} \left[ Z,E_p \right] + U(Z,E_p) = U(Z,E_p) =  \\
    &= \frac{1}{2} B \left( Z , \left[F_p,E_p\right] \right) F_p = - \frac{1}{2} a_p F_p  \, , \\
    \Uplambda \left( Z \right) (F_p) &= \frac{1}{2} \left[ Z,F_p \right] + U(Z,F_p) = U(Z,F_p) =  \\
    &= \frac{1}{2} B \left( Z , \left[E_p,F_p\right] \right) E_p = \frac{1}{2} a_p E_p \, , \\
    \Uplambda \left( Z \right) (Z) &= \frac{1}{2} \left[ Z,Z \right] + U(Z,Z) = 0 \, . \\
\end{align*}
This concludes the proof. 
\end{proof}

Recalling the notation of Section \ref{sec:diff_forms_approach}, we can now state the following result:  

\begin{thmm}\label{thmm:main}
The invariant spinor on $\left(H_{2n+1},g \right)$ determined by $1 \in \Sigma$ is a generalised Killing spinor with the following set of eigenvalues: 
\[ 
\left\{ \lambda_p = \frac{1}{4} a_p \, , \,  \mu = - \sum_{q=1}^{n} \lambda_q \right\}_{p=1 , \dots , n} \, . 
\]
In particular, choosing $a_p \coloneqq p$, there are $n+1$ distinct constant eigenvalues. 
\end{thmm}
\begin{proof}
For $X \in \mathfrak{h}_{2n+1}$, let $\widetilde{\Uplambda}(X) \in \mathfrak{spin}(2n+1)$ be the image of $\Uplambda(X) \in \mathfrak{so}(2n+1)$ under the isomorphism 
\[ 
    e_i \wedge e_j \mapsto \frac{1}{2} e_i \cdot e_j \, . 
\] 
We will check that, for all $1 \leq p \leq n$,  
\[ \widetilde{\Uplambda}\left( E_p \right) \cdot 1 = \lambda_p E_p \cdot 1 \, , \quad \widetilde{\Uplambda}\left( F_p \right) \cdot 1 = \lambda_p F_p \cdot 1 \, , \quad \widetilde{\Uplambda}\left( Z \right) \cdot 1 = \mu Z \cdot 1 \, . \] 
From Lemma \ref{lemma:h_nomizu}, we get that the spin lift $\widetilde{\Uplambda}$ of the Nomizu map $\Uplambda$ is given by 
\[\widetilde{\Uplambda} \left( E_{p} \right) = \frac{1}{4} a_p F_p \cdot Z \, , \quad \widetilde{\Uplambda} \left( F_{p} \right) = -\frac{1}{4} a_p E_p \cdot Z \, , \quad \widetilde{\Uplambda} \left( Z \right) = -\frac{1}{4} \sum_{p=1}^{n} a_p E_p \cdot F_p \, . \] 
Hence, applying these operators to the spinor $1 \in \Sigma$, we obtain 
\begin{align*}
    & \widetilde{\Uplambda} \left( E_{p} \right) \cdot 1 = \frac{1}{4} a_p F_p \cdot Z \cdot 1 = \frac{i}{4} a_p F_p \cdot 1 = \frac{i}{4} a_p y_p \, , \\
    & \widetilde{\Uplambda} \left( F_{p} \right) \cdot 1 = -\frac{1}{4} a_p E_p \cdot Z \cdot 1 = -\frac{i}{4} a_p E_p \cdot 1 = \frac{1}{4} a_p y_p  \, , \\
    & \widetilde{\Uplambda} \left( Z \right) \cdot 1 = -\frac{1}{4} \sum_{p=1}^{n} a_p E_p \cdot F_p \cdot 1 = -\frac{i}{4} \sum_{p=1}^{n} a_p \cdot 1 \, . 
    \end{align*}
And because  
\[ 
    E_p \cdot 1 =  i y_p \, , \qquad F_p \cdot 1 = y_p \, , \qquad Z \cdot 1 = i 1 \, , 
\] 
we conclude. 
\end{proof} 

\section{Three-dimensional Lie groups}\label{sec:3dim}

In this section, we study the existence of left-invariant metrics on three-dimensional Lie groups which admit invariant generalised Killing spinors. We use Bianchi's classification of three-dimensional real Lie algebras \cite{bianchi,M63} --see Appendix \ref{appendix}. 

For a general spin manifold, the set of generalised Killing spinors does not have a good algebraic structure. It is clear that this set is closed under multiplication by scalars in $\C$, but in general it is not closed under addition. Not even the set of \textit{invariant} generalised Killing spinors on a homogeneous space can be endowed with a good structure. However, in the case of connected three-dimensional Lie groups, this set is a complex vector space, as we show in the following lemma. 

\begin{lemma} \label{lemma:allornothing}
Let $G$ be a connected three-dimensional Lie group with Lie algebra $\mathfrak{g}$, and let $g$ be a left-invariant metric on $G$ carrying a non-zero invariant generalised Killing spinor with associated endomorphism $A$. Then, every invariant spinor is generalised Killing with endomorphism $A$. 
\end{lemma}
\begin{proof}
Let $\left(e_1,e_2,e_3\right)$ be a $g$-orthonormal basis of $\mathfrak{g}$. Let $\widetilde{\Uplambda}$ be the lift to $\mathfrak{spin}(3)$ of the Nomizu map of the Levi-Civita connection of $g$. There exist linear forms
\[ \lambda,\mu,\nu \colon \mathfrak{g} \to \mathbb{R} \]
such that, for every $X\in\mathfrak{g}$,  
\[
\widetilde{\Uplambda}(X) = \lambda(X) e_1 e_2 + \mu(X) e_1 e_3 + \nu(X) e_2 e_3 \, , 
\]
and hence
\begin{align*}
\widetilde{\Uplambda}(X) \cdot 1 &= \lambda(X) e_1 e_2 \cdot 1 + \mu(X) e_1 e_3 \cdot 1 + \nu(X) e_2 e_3 \cdot 1 \\ 
&= i \lambda(X) e_1 \cdot \left( y_1 \right) + \mu(X) e_1 \cdot \left( y_1 \right) + \nu(X) e_2 \cdot \left( y_1 \right) \\
&= \lambda(X) y_1 - i \mu(X) y_1 + i \nu(X) 1 \\
&= \lambda(X) e_3 \cdot 1 - \mu(X) e_2 \cdot 1 + \nu(X) e_1 \cdot 1 \\
&= \left( \lambda(X) e_3 - \mu(X) e_2 + \nu(X) e_1 \right) \cdot 1 \, . 
\end{align*}
Similarly, 
\begin{align*}
    \widetilde{\Uplambda}(X) \cdot y_1 &= \lambda(X) e_1 e_2 \cdot y_1 + \mu(X) e_1 e_3 \cdot y_1 + \nu(X) e_2 e_3 \cdot y_1 \\ 
    &= i \lambda(X) e_1 \cdot \left( 1 \right) - \mu(X) e_1 \cdot \left( 1\right) - \nu(X) e_2 \cdot \left( 1 \right) \\
    &= -\lambda(X) 1 - i \mu(X) 1 - i \nu(X) y_1 \\
    &= \lambda(X) e_3 \cdot y_1 - \mu(X) e_2 \cdot y_1 + \nu(X) e_1 \cdot y_1 \\
    &= \left( \lambda(X) e_3 - \mu(X) e_2 + \nu(X) e_1 \right) \cdot y_1 \, . 
\end{align*}
Define $B \in \End(\mathfrak{g})$ by 
\[ X \mapsto \lambda(X) e_3 - \mu(X) e_2 + \nu(X) e_1 \, . \]
It is clear that, for each $X \in \mathfrak{g}$,  
and every invariant spinor $\phi \in \Sigma$, 
\begin{equation} \label{eq:allgks}
     \widetilde{\Uplambda}(X) \cdot \phi = B(X) \cdot \phi \, .
\end{equation}
But, by hypothesis, there exists a non-zero invariant generalised Killing spinor $\psi$ with associated symmetric endomorphism $A \in \End(\mathfrak{g})$. Hence, for every $X \in \mathfrak{g}$, 
\[ A(X) \cdot \psi = \widetilde{\Uplambda}(X) \cdot \psi = B(X) \cdot \psi \, .  \]
As $\psi \neq 0$, this implies that $A=B$, and hence $B$ is symmetric. So, by \eqref{eq:allgks}, every invariant spinor is generalised Killing with endomorphism $A$. 
\end{proof}

We can now state the main result of this section, which gives a complete picture of the set of invariant generalised Killing spinors on connected three-dimensional Lie groups. 

\begin{thmm}\label{thm:3dim}
Let $G$ be a connected three-dimensional Lie group with Lie algebra $\mathfrak{g}$. For any choice of orientation and left-invariant metric on $G$, the following holds: 
\begin{enumerate}
    \item If $\mathfrak{g} = \mathrm{Bian}(\text{I})$, then every left-invariant metric on $G$ carries a two-dimensional space of invariant generalised Killing spinors with one eigenvalue (in fact, the spinors are parallel). 
    \item If $\mathfrak{g} = \mathrm{Bian}(\text{II})$, then every left-invariant metric on $G$ carries a two-dimensional space of invariant generalised Killing spinors with two distinct eigenvalues. 
    \item If $\mathfrak{g} = \mathrm{Bian}(\text{III})$, $\mathrm{Bian}(\text{IV})$, $\mathrm{Bian}(\text{V})$, $\mathrm{Bian}(\text{VI}_{x})$ for some $0 < \abs{x} < 1$ or $\mathrm{Bian}(\text{VII}_{y})$ for some $y > 0$, then no left-invariant metric on $G$ carries non-trivial invariant generalised Killing spinors. 
    \item If $\mathfrak{g} = \mathrm{Bian}(\text{VI}_{-1})$, then every left-invariant metric on $G$ carries a two-dimensional space of invariant generalised Killing spinors with three distinct eigenvalues.
    \item If $\mathfrak{g} = \mathrm{Bian}(\text{VII}_0)$, $\mathrm{Bian}(\text{VIII})$ or $\mathrm{Bian}(\text{IX})$, then every left-invariant metric on $G$ carries a two-dimensional space of invariant generalised Killing spinors. Moreover, the set of left-invariant metrics for which the number of distinct eigenvalues of these spinors is strictly less than $3$ has measure zero in the set of all left-invariant metrics on $G$. 
\end{enumerate}
This is summarised in Table \ref{table:3-dim}. 
\end{thmm} 

\begin{proof} 
    The strategy in every case will be the following. We start by fixing a basis $(f_1,f_2,f_3)$ of $\mathfrak{g}$ for which the structure constants take the form given in Appendix \ref{appendix}. We fix the orientation defined by this ordered basis (a different choice of orientation does not affect the subsequent results). Then, we consider an invertible matrix of the form 
    \[ 
        P = \begin{pmatrix}
            \alpha & \beta  & \gamma \\
             0 & \varepsilon & \zeta \\
             0 & 0 & \iota \\
        \end{pmatrix} \in \GL^{+}(3,\mathbb{R}) \, , 
    \]
    and define $B_P$ to be the unique inner product on $\mathfrak{g}$ such that the basis $\left(e_1,e_2,e_3\right)$ defined by 
    \[
    e_1 = \alpha f_1 \, , \quad e_2 = \beta f_1 + \varepsilon f_2 \, , \quad \text {and} \quad e_3 = \gamma f_1 + \zeta f_2 + \iota f_3  
    \]
    is (positively oriented and) $B_P$-orthonormal. Note that this covers all possible inner products on $\mathfrak{g}$. Then, we compute the endomorphism $A$ of $\mathfrak{g}$ such that, for every $X \in \mathfrak{g}$, 
    \begin{equation} \label{eq:genkil3dim} 
    \widetilde{\Uplambda}(X) \cdot 1 = A(X) \cdot 1 \, . 
    \end{equation}
    This endomorphism exists by the proof of Lemma \ref{lemma:allornothing}, and it is obtained after a computation analogous to the one done in detail in Section \ref{sec:heisenberg}. We give a general expression later, in equation \eqref{eq:explicit_A}. The explicit results in terms of $P$ are collected in Appendix \ref{appendix}. 
    
    Finally, by Lemma \ref{lemma:allornothing}, there are only two cases: if $A$ is symmetric, then all invariant spinors are generalised Killing with endomorphism $A$; if not, the only invariant generalised Killing spinor is $0$. In Table \ref{table:symmetric}, we show the value of $A-A^T$ in each case --see also Remark \ref{rem:symmetry}. 

    \begin{table}[h!]
        \centering
        \begin{tabular}{cc}
            \toprule 
            $\mathfrak{g}$ & $A - A^T$ \\
            \toprule
            $\mathrm{Bian}(\text{I})$, $\mathrm{Bian}(\text{II})$, $\mathrm{Bian}(\text{VIII})$, $\mathrm{Bian}(\text{IX})$  & $\begin{pmatrix} 0 & 0 & 0 \\ 0 & 0 & 0 \\ 0 & 0 & 0 \end{pmatrix}$ \\
            \midrule
            $\mathrm{Bian}(\text{III})$ & $ \frac{1}{2}\begin{pmatrix}
                0 & -\zeta  & \varepsilon \\ 
                \zeta & 0 & 0 \\
                -\varepsilon & 0 & 0 \\
                \end{pmatrix}$ \\
            \midrule
            $\mathrm{Bian}(\text{IV})$, $\mathrm{Bian}(\text{V})$ & $\begin{pmatrix}
                0 & -\iota  & 0 \\
                \iota & 0 & 0 \\
                0 & 0 & 0 \\ 
                \end{pmatrix}$ \\
            \midrule
            $\mathrm{Bian}(\text{VI}_{x})$ & $\frac{x+1}{2} \begin{pmatrix}
                0 & -\iota  & 0 \\
                \iota & 0 & 0 \\
                0 & 0 & 0 \\
                \end{pmatrix}$ \\
            \midrule
            $\mathrm{Bian}(\text{VII}_{y})$ & $y \begin{pmatrix}
                0 & -\iota  & 0 \\
                \iota & 0 & 0 \\
                0 & 0 & 0 \\
                \end{pmatrix}$ \\
            \bottomrule
        \end{tabular}
        \caption{For each Lie algebra $\mathfrak{g}$, value of $A-A^T$ in the basis $(e_1,e_2,e_3)$.}
        \label{table:symmetric}
    \end{table}
\end{proof}

\begin{rem}\label{rem:symmetry} It is very surprising that the symmetry of $A$ is completely independent of the choice of left-invariant metric -- see Table \ref{table:symmetric}. 
\end{rem} 

In fact, we can obtain a general expression for the endomorphism $A$ as follows. Fix a connected three-dimensional Lie group $G$, and fix an orientation and an inner product $B$ on its Lie algebra $\mathfrak{g}$. Let $(e_1,e_2,e_3)$ be a positively oriented $B$-orthonormal basis of $\mathfrak{g}$. Then, for each $1 \leq i,j \leq 3$, 
\[ \left[ e_i , e_j \right] = \sum_{k=1}^{3} c_{ij}^{k} e_k \, , \] 
for some $c_{ij}^{k} \in \R$. A calculation similar to the one detailed in Section \ref{sec:heisenberg} yields the matrix of $A$ in the basis $(e_1,e_2,e_3)$:  
\begin{equation}\label{eq:explicit_A}
    A = \begin{pmatrix}
    \frac{1}{4} (c_{12}^{3}-c_{13}^{2}-c_{23}^{1}) & -\frac{1}{2} c_{23}^{2} & -\frac{1}{2} c_{23}^{3} \\
    \frac{1}{2} c_{13}^{1} & \frac{1}{4} (c_{12}^{3}+c_{13}^{2}+c_{23}^{1}) & \frac{1}{2} c_{13}^{3} \\
    -\frac{1}{2} c_{12}^{1} & -\frac{1}{2} c_{12}^{2} & \frac{1}{4} (-c_{12}^{3}-c_{13}^{2}+c_{23}^{1}) \\
   \end{pmatrix} \, . 
\end{equation}
This matrix is symmetric if, and only if, 
\[ c_{13}^{1} + c_{23}^{2} = 0 \, , \quad c_{12}^{1} - c_{23}^{3} = 0 \, , \quad \text {and} \quad c_{12}^{2} + c_{13}^{3} = 0 \, . \] 
In this case, every invariant spinor $\psi \in \Sigma$ is a generalised Killing spinor with endomorphism $A$. In particular, every $\psi \in \Sigma$ is an eigenspinor of the Dirac operator with eigenvalue 
\[ \tr(A) = \frac{1}{4} (c_{12}^{3} - c_{13}^{2} + c_{23}^{1}) \, . \] 
The endomorphism $A$ commutes with the Ricci endomorphism, as shown in the following: 
\begin{thmm} \label{thm:ricci_A} 
Let $G$ be a connected three-dimensional Lie group equipped with an orientation and a left-invariant metric $g$ carrying a non-trivial invariant generalised Killing spinor with endomorphism $A$. Then, $A$ commutes with the Ricci endomorphism of $g$.  
\end{thmm}
\begin{proof}
    Let $\psi$ be such a spinor, and let $\Ric$ be the Ricci endomorphism of $g$. As, by hypothesis, $A$ is symmetric, using the well-known expression 
    \[
    \sum_{j=1}^{3} e_j \mathrm{R}(X,e_j) \cdot \psi = -\frac{1}{2} \Ric(X) \cdot \psi \, , 
    \]
    and the explicit expression of $A$ in \eqref{eq:explicit_A}, one obtains the matrix of $\Ric$ in the basis $(e_1,e_2,e_3)$: 
    \begin{equation} \label{eq:ricci_matrix}
        \resizebox{\textwidth}{!}{
            $\begin{pmatrix}
                \frac{1}{2} \left(-(c_{12}^{3})^2-(c_{13}^{2})^2-4 (c_{13}^{3})^2+(c_{23}^{1})^2\right) & -(c_{12}^{3}+c_{13}^{2}-c_{23}^{1}) c_{23}^{2}-2 c_{13}^{3} c_{23}^{3} & (c_{12}^{3}+c_{13}^{2}+c_{23}^{1}) c_{23}^{3}-2 c_{13}^{3} c_{23}^{2} \\
                -(c_{12}^{3}+c_{13}^{2}-c_{23}^{1}) c_{23}^{2}-2 c_{13}^{3} c_{23}^{3} & \frac{1}{2} \left((c_{13}^{2})^2-(c_{12}^{3}-c_{23}^{1})^2-4 (c_{23}^{3})^2\right) & c_{13}^{3} (-c_{12}^{3}+c_{13}^{2}+c_{23}^{1})+2 c_{23}^{2} c_{23}^{3} \\
                (c_{12}^{3}+c_{13}^{2}+c_{23}^{1}) c_{23}^{3}-2 c_{13}^{3} c_{23}^{2} & c_{13}^{3} (-c_{12}^{3}+c_{13}^{2}+c_{23}^{1})+2 c_{23}^{2} c_{23}^{3} & \frac{1}{2} \left((c_{12}^{3})^2-(c_{13}^{2}+c_{23}^{1})^2-4 (c_{23}^{2})^2\right) \\
               \end{pmatrix}$ \, . 
        }
        \end{equation}
        
    Using the Jacobi identity, one can easily check that $\Ric$ and $A$ commute. 
\end{proof}

\begin{rem} \label{rem:ricci}
    As a consequence of Theorem \ref{thm:ricci_A}, there exists a global orthonormal frame of left-invariant vector fields on $G$ which are eigenvectors of both $A$ and $\Ric$. This illustrates one of the main features of spin geometry, namely that it reduces curvature calculations from a second-order problem to a first-order one.
\end{rem} 

For example, consider the case of the Heisenberg group $H_3$, which has Lie algebra $\mathfrak{h}_3 \cong \mathrm{Bian}(\text{II})$, with an arbitrary orientation and left-invariant metric. Take $\{Z , E , F \}$ an orthonormal basis of $\mathfrak{h}_3$ as in \eqref{eq:basis_h_2}, satisfying $[E,F]=a Z , \, [E,Z]=[F,Z]=0$, for some $a >0$. Identify now the Lie algebra with the space of left-invariant vector fields in the usual way. We saw -- see the proof of Theorem \ref{thmm:main} -- that, in this case, the endomorphism $A$ is symmetric and has two distinct eigenvalues, which give rise to two left-invariant subbundles $\Xi = \vecspan\{Z\}$ and $\Xi^{\perp}= \vecspan\{E,F\}$ of the tangent bundle of $H_3$. The distribution $\Xi^{\perp}$ defines a contact metric structure on $H_3$, and $\Xi$ is the corresponding Reeb distribution, because $[E,F] \notin \Xi$. The matrix of the Ricci endomorphism in this basis is obtained by substituting the structure constants into \eqref{eq:ricci_matrix}, yielding 
\[
\frac{a^2}{2}\begin{pmatrix}
 1 & 0 & 0 \\
 0 & -1 & 0 \\
 0 & 0 & -1 \\
\end{pmatrix} \, , 
\]
which shows that any left-invariant metric on $H_3$ is $\eta$-Einstein. 

\begin{appendices}
  \addtocontents{toc}{\protect\setcounter{tocdepth}{0}}

  \section{Explicit matrices} \label{appendix}
  
  Let $G$ be a connected three-dimensional Lie group with Lie algebra $\mathfrak{g}$. 
  \[ \mathfrak{g} = \vecspan_{\mathbb{R}} \{ f_1 , f_2 , f_3 \} \, . \]  
  A left-invariant metric on $G$ corresponds to an inner product $B$ on $\mathfrak{g}$. By applying the Gram-Schmidt process to the oriented basis $(f_1,f_2,f_3)$, we obtain an oriented $B$-orthonormal basis $(e_1,e_2,e_3)$. More specifically, there exists a matrix 
  \[ 
      P = \begin{pmatrix}
          \alpha & \beta  & \gamma \\
           0 & \varepsilon & \zeta \\
           0 & 0 & \iota \\
      \end{pmatrix} \in \GL^{+}(3,\mathbb{R})   
  \]
  such that 
  \[
      e_1 = \alpha f_1 \, , \quad e_2 = \beta f_1 + \varepsilon f_2 \, , \quad \text {and} \quad e_3 = \gamma f_1 + \zeta f_2 + \iota f_3 \, . 
      \]
  We saw in the proof of Lemma \ref{lemma:allornothing} that there exists $A \in \End(\mathfrak{g})$ such that, for all invariant spinors $\phi \in \Sigma$ and all $X \in \mathfrak{g}$, 
  \[ 
  \widetilde{\Uplambda}(X) \cdot \phi = A(X) \cdot \phi \, . 
  \]
  For each three-dimensional Lie algebra in Bianchi's list \cite{bianchi}, we give an explicit expression of the matrix of $A$ in the basis $(e_1 , e_2 , e_3)$. These matrices are obtained by following the same procedure we illustrated in Section \ref{sec:heisenberg}. We used Wolfram Mathematica to perform symbolic computations and solve linear systems of equations that depend on parameters. To ensure the reliability of the solutions, we verified them manually to address any potential limitations associated with the use of \texttt{Solve}. We use the structure constants given in \cite{M63}. 

  \subsection{$\mathrm{Bian}(\text{I})$}

  $\mathrm{Bian}(\text{I})$ is the abelian three-dimensional Lie algebra, whose associated matrix $A$ is clearly the zero matrix. 

  \subsection{$\mathrm{Bian}(\text{II})$}
  
  The Lie bracket is given by 
  \begin{center}
      \begin{tabular}{c|ccc}
          & $f_1$ & $f_2$ & $f_3$ \\
       \hline
       $f_1$ & $0$ & $0$ & $0$ \\
       $f_2$ & $0$ & $0$ & $f_1$ \\
       $f_3$ & $0$ & $-f_1$ & $0$ \\
       \end{tabular}
  \end{center}
  meaning that $[f_1,f_2]=[f_1,f_3]=0$, $[f_2,f_3]=f_1$, and so on. This is the three-dimensional Heisenberg algebra $\mathfrak{h}_3$. In this case, 
  \[ A = \frac{\det(P)}{4\alpha^2} \begin{pmatrix}
      - 1 & 0 & 0 \\
      0 & 1 & 0 \\
      0 & 0 & 1 \\
  \end{pmatrix} \, , 
  \]
  As $\det(P) \neq 0$, this matrix always has exactly two distinct eigenvalues.  
  
  \subsection{$\mathrm{Bian}(\text{III})$}
  
  The Lie bracket is given by 
  \begin{center}
      \begin{tabular}{c|ccc}
          & $f_1$ & $f_2$ & $f_3$ \\
       \hline
       $f_1$ & $0$ & $f_1$ & $0$ \\
       $f_2$ & $-f_1$ & $0$ & $0$ \\
       $f_3$ & $0$ & $0$ & $0$ \\
       \end{tabular}
      \end{center}
  
  This is a direct sum of the only non-abelian two-dimensional Lie algebra and the (trivial) one-dimensional one. Then, 
  \[ A = \frac{1}{4 \alpha} \begin{pmatrix}
      \gamma \varepsilon - \beta \zeta & 0 & 0 \\
      2 \alpha \zeta & \beta \zeta - \gamma \varepsilon & 0 \\
      -2 \alpha \varepsilon & 0 & \beta \zeta - \gamma \varepsilon \\
      \end{pmatrix} \, , 
  \]
  which is not symmetric, as $P$ is non-singular.  
  
  \subsection{$\mathrm{Bian}(\text{IV})$}
  
  The Lie bracket is given by
  \begin{center}
      \begin{tabular}{c|ccc}
          & $f_1$ & $f_2$ & $f_3$ \\
       \hline
       $f_1$ & $0$ & $0$ & $f_1$ \\
       $f_2$ & $0$ & $0$ & $f_1 + f_2$ \\
       $f_3$ & $-f_1$ & $-f_1 - f_2$ & $0$ \\
       \end{tabular}
      \end{center}
  
  In this case, 
  \[ A = \frac{1}{4 \alpha}\begin{pmatrix}
      - \iota \varepsilon & -2 \alpha \iota & 0 \\
      2 \alpha \iota & \iota \varepsilon & 0 \\
      0 & 0 & \iota \varepsilon \\
      \end{pmatrix} \, . 
  \]
  The condition $\det(P) \neq 0$ implies that $A$ is not symmetric.
  
  \subsection{$\mathrm{Bian}(\text{V})$}
  
  We denote by $\mathrm{Bian}(\text{V})$ the Lie algebra with Lie bracket given by 
  \begin{center}
      \begin{tabular}{c|ccc}
          & $f_1$ & $f_2$ & $f_3$ \\
       \hline
       $f_1$ & $0$ & $0$ & $f_1$ \\
       $f_2$ & $0$ & $0$ & $f_2$ \\
       $f_3$ & $-f_1$ & $- f_2$ & $0$ \\
       \end{tabular}
      \end{center}
  The matrix $A$ in this case is given by 
  \[ A = \begin{pmatrix}
      0 & -\frac{\iota}{2} & 0 \\
      \frac{\iota}{2} & 0 & 0 \\
      0 & 0 & 0 \\
     \end{pmatrix}
      \, ,   
  \]
  which is not symmetric, as $\iota \neq 0$. 

  \subsection{$\mathrm{Bian}(\text{VI}_{x})$, $-1 \leq x < 1$, $x \neq 0$}
  
  For each $-1 \leq x < 1$, $x \neq 0$, we denote by $\mathrm{Bian}(\text{VI}_{x})$ the Lie algebra with Lie bracket given by 
  \begin{center}
      \begin{tabular}{c|ccc}
          & $f_1$ & $f_2$ & $f_3$ \\
       \hline
       $f_1$ & $0$ & $0$ & $f_1$ \\
       $f_2$ & $0$ & $0$ & $x f_2$ \\
       $f_3$ & $-f_1$ & $-x f_2$ & $0$ \\
       \end{tabular}
      \end{center}
  
  Note that, for $x = -1$, this is the Poincar\'{e} algebra $\mathfrak{p}(1,1)$, \textit{i.e.}, the Lie algebra of the Poincar\'{e} group $P(1,1)$ --see \cite[Section 1.2.5]{hall}. For general $x$, 
  \[ A = \begin{pmatrix}
      \frac{(x-1) \beta \iota}{4 \alpha} & -\frac{\iota x}{2} & 0 \\
      \frac{\iota}{2} & -\frac{(x-1) \beta \iota}{4 \alpha} & 0 \\
      0 & 0 & -\frac{(x-1) \beta \iota}{4 \alpha} \\
     \end{pmatrix}
      \, .  
  \]
  This matrix is symmetric if, and only if, $x = -1$. In this case, its eigenvalues are 
  \[ 
      \frac{\beta \iota}{2 \alpha}\, , \quad \pm \frac{\sqrt{\alpha^2 \iota^2 + \beta^2 \iota^2}}{2 \alpha}\, , 
  \]
  which are always different.  
  
  \subsection{$\mathrm{Bian}(\text{VII}_{y})$, $y \geq 0$} 
  
  For $y \geq 0$, the Lie bracket is given by
  \begin{center}
      \begin{tabular}{c|ccc}
          & $f_1$ & $f_2$ & $f_3$ \\
       \hline
       $f_1$ & $0$ & $0$ & $y f_1 - f_2$ \\
       $f_2$ & $0$ & $0$ & $f_1 + y f_2$ \\
       $f_3$ & $-y f_1 + f_2$ & $-f_1 -y f_2$ & $0$ \\
       \end{tabular}
      \end{center}
  
  For $y=0$, this is the Euclidean algebra $\mathfrak{e}(2)$, \textit{i.e.}, the Lie algebra of the Lie group $E(2)$ --see \cite[Section 1.2.5]{hall}. For general $y \geq 0$, 
  \[ A = \frac{1}{4 \det(P)}\begin{pmatrix}
      \iota^2 (\alpha^2-\beta^2-\varepsilon^2) & 2 \alpha \iota^2 (\beta-\varepsilon y) & 0 \\
      2 \alpha \iota^2 (\beta+\varepsilon y) & \iota^2 (-\alpha^2+\beta^2+\varepsilon^2) & 0 \\
      0 & 0 & \iota^2 (\alpha^2+\beta^2+\varepsilon^2) \\
  \end{pmatrix}\, , 
  \] 
  which is symmetric if, and only if, $y=0$. In this case, $A$ has eigenvalues 
  \[ \lambda \, , \quad \pm \sqrt{\lambda^2 - \frac{1}{4} \iota^2} \, , \qquad \text{where} \quad \lambda = \frac{1}{4 \det(P)} \iota^2 \left( \alpha^2 + \beta^2 + \varepsilon^2 \right) \, . \] 
  These eigenvalues are generically different. 
  
  \subsection{$\mathrm{Bian}(\text{VIII})$} 

  The Lie bracket is given by
  \begin{center}
      \begin{tabular}{c|ccc}
          & $f_1$ & $f_2$ & $f_3$ \\
       \hline
       $f_1$ & $0$ & $f_1$ & $2 f_2$ \\
       $f_2$ & $- f_1$ & $0$ & $f_3$ \\
       $f_3$ & $-2 f_2$ & $-f_3$ & $0$ \\
       \end{tabular}
      \end{center}
  
      This is the Lie algebra $\mathfrak{sl}(2,\mathbb{R})$. In this case,  
  \[ A = \frac{1}{2 \det(P)} \begin{pmatrix}
      a_{11} & a_{12} & a_{13} \\
      a_{21} & a_{22} & a_{23} \\
      a_{31} & a_{32} & a_{33} \\
  \end{pmatrix} \, , 
  \]
  where 
  \begin{align*}
    &a_{11} =  \iota (-\alpha^2 \iota + \beta^2 \iota + \gamma \varepsilon^2 - \beta \varepsilon \zeta) \, , \\
    &a_{12} = \alpha \iota (-2 \beta \iota + \varepsilon \zeta) \, , \\
    &a_{13} = -\alpha \iota \varepsilon^2  \, , \\
    &a_{21} = \alpha \iota (-2 \beta \iota + \varepsilon \zeta) \, , \\
    &a_{22} = \iota (\alpha^2 \iota - \beta^2 \iota - \gamma \varepsilon^2 + \beta \varepsilon \zeta)  \, , \\ 
    &a_{23} = 0  \, , \\
    &a_{31} = -\alpha \iota \varepsilon^2 \, , \\
    &a_{32} = 0 \, , \\
    &a_{33} = -\iota ((\alpha^2 + \beta^2) \iota + \gamma \varepsilon^2 - \beta \varepsilon \zeta) \, .
\end{align*}
  This matrix is symmetric, and has generically three distinct eigenvalues. 
  
  \subsection{$\mathrm{Bian}(\text{IX})$} 
  
  The Lie bracket is given by
  \begin{center}
      \begin{tabular}{c|ccc}
          & $f_1$ & $f_2$ & $f_3$ \\
       \hline
       $f_1$ & $0$ & $f_3$ & $-f_2$ \\
       $f_2$ & $-f_3$ & $0$ & $f_1$ \\
       $f_3$ & $f_2$ & $-f_1$ & $0$ \\
       \end{tabular}
      \end{center}
  
  This is the Lie algebra $\mathfrak{sp}(1) \cong \mathfrak{su}(2) \cong \mathfrak{so}(3)$. In this case, 
  \[ A = \frac{1}{4 \det(P)} \begin{pmatrix}
      a_{11} & a_{12} & a_{13} \\
      a_{21} & a_{22} & a_{23} \\
      a_{31} & a_{32} & a_{33} \\
  \end{pmatrix} \, , 
  \]
  where 
  \begin{align*}
      &a_{11} = \alpha^2 \left(\iota^2+\varepsilon^2+\zeta^2\right)-\iota^2 \left(\beta^2+\varepsilon^2\right)-(\gamma \varepsilon-\beta \zeta)^2 \, , \\
      &a_{12} = 2 \alpha \left(\beta \left(\iota^2+\zeta^2\right)-\gamma \varepsilon \zeta\right) \, , \\
      &a_{13} = 2 \alpha \varepsilon (\gamma \varepsilon-\beta \zeta) \, , \\
      &a_{21} = 2 \alpha \left( \beta \left(\iota^2+\zeta^2\right)-\gamma \varepsilon \zeta \right) \, , \\
      &a_{22} = -\alpha^2 \left(\iota^2-\varepsilon^2+\zeta^2\right)+\iota^2 \left(\beta^2+\varepsilon^2\right)+(\gamma \varepsilon-\beta \zeta)^2 \, , \\ 
      &a_{23} = 2 \alpha^2 \varepsilon \zeta \, , \\
      &a_{31} = 2 \alpha \varepsilon (\gamma \varepsilon-\beta \zeta) \, , \\
      &a_{32} = 2 \alpha^2 \varepsilon \zeta \, , \\
      &a_{33} = \alpha^2 \left(\iota^2-\varepsilon^2+\zeta^2\right)+\iota^2 \left(\beta^2+\varepsilon^2\right)+(\gamma \varepsilon-\beta \zeta)^2 \, .
  \end{align*}
  This matrix is symmetric, and has generically three distinct eigenvalues. 
  
\end{appendices}

\FloatBarrier

\section*{Acknowledgements}
The author would like to thank his supervisor, Marie-Am\'{e}lie Lawn, and Ilka Agricola for their comments on the first draft of this paper, as well as the anonymous referees for their insightful comments. The author is funded by the UK Engineering and Physical Sciences Research Council (EPSRC), grant EP/W5238721.

\bibliographystyle{alphaurl}
\bibliography{references.bib}

\end{document}